\documentclass[12pt]{amsart}
\usepackage{amsmath,amssymb,amsbsy,amsfonts,latexsym,amsopn,amstext,
                                               amsxtra,euscript,amscd}
\usepackage{color}
\usepackage{graphicx}

\usepackage{float}

\newfont{\teneufm}{eufm10}
\newfont{\seveneufm}{eufm7}
\newfont{\fiveeufm}{eufm5}
\newfam\eufmfam
                \textfont\eufmfam=\teneufm \scriptfont\eufmfam=\seveneufm
                \scriptscriptfont\eufmfam=\fiveeufm
\def\bbbc{{\mathchoice {\setbox0=\hbox{$\displaystyle\rm C$}\hbox{\hbox
to0pt{\kern0.4\wd0\vrule height0.9\ht0\hss}\box0}}
{\setbox0=\hbox{$\textstyle\rm C$}\hbox{\hbox
to0pt{\kern0.4\wd0\vrule height0.9\ht0\hss}\box0}}
{\setbox0=\hbox{$\scriptstyle\rm C$}\hbox{\hbox
to0pt{\kern0.4\wd0\vrule height0.9\ht0\hss}\box0}}
{\setbox0=\hbox{$\scriptscriptstyle\rm C$}\hbox{\hbox
to0pt{\kern0.4\wd0\vrule height0.9\ht0\hss}\box0}}}}
\def\bbbq{{\mathchoice {\setbox0=\hbox{$\displaystyle\rm
Q$}\hbox{\raise 0.15\ht0\hbox to0pt{\kern0.4\wd0\vrule
height0.8\ht0\hss}\box0}} {\setbox0=\hbox{$\textstyle\rm
Q$}\hbox{\raise 0.15\ht0\hbox to0pt{\kern0.4\wd0\vrule
height0.8\ht0\hss}\box0}} {\setbox0=\hbox{$\scriptstyle\rm
Q$}\hbox{\raise 0.15\ht0\hbox to0pt{\kern0.4\wd0\vrule
height0.7\ht0\hss}\box0}} {\setbox0=\hbox{$\scriptscriptstyle\rm
Q$}\hbox{\raise 0.15\ht0\hbox to0pt{\kern0.4\wd0\vrule
height0.7\ht0\hss}\box0}}}}
\def\bbbt{{\mathchoice {\setbox0=\hbox{$\displaystyle\rm
T$}\hbox{\hbox to0pt{\kern0.3\wd0\vrule height0.9\ht0\hss}\box0}}
{\setbox0=\hbox{$\textstyle\rm T$}\hbox{\hbox
to0pt{\kern0.3\wd0\vrule height0.9\ht0\hss}\box0}}
{\setbox0=\hbox{$\scriptstyle\rm T$}\hbox{\hbox
to0pt{\kern0.3\wd0\vrule height0.9\ht0\hss}\box0}}
{\setbox0=\hbox{$\scriptscriptstyle\rm T$}\hbox{\hbox
to0pt{\kern0.3\wd0\vrule height0.9\ht0\hss}\box0}}}}
\def\bbbs{{\mathchoice
{\setbox0=\hbox{$\displaystyle     \rm S$}\hbox{\raise0.5\ht0\hbox
to0pt{\kern0.35\wd0\vrule height0.45\ht0\hss}\hbox
to0pt{\kern0.55\wd0\vrule height0.5\ht0\hss}\box0}}
{\setbox0=\hbox{$\textstyle        \rm S$}\hbox{\raise0.5\ht0\hbox
to0pt{\kern0.35\wd0\vrule height0.45\ht0\hss}\hbox
to0pt{\kern0.55\wd0\vrule height0.5\ht0\hss}\box0}}
{\setbox0=\hbox{$\scriptstyle      \rm S$}\hbox{\raise0.5\ht0\hbox
to0pt{\kern0.35\wd0\vrule height0.45\ht0\hss}\raise0.05\ht0\hbox
to0pt{\kern0.5\wd0\vrule height0.45\ht0\hss}\box0}}
{\setbox0=\hbox{$\scriptscriptstyle\rm S$}\hbox{\raise0.5\ht0\hbox
to0pt{\kern0.4\wd0\vrule height0.45\ht0\hss}\raise0.05\ht0\hbox
to0pt{\kern0.55\wd0\vrule height0.45\ht0\hss}\box0}}}}
\def\bbbz{{\mathchoice {\hbox{$\sf\textstyle Z\kern-0.4em Z$}}
{\hbox{$\sf\textstyle Z\kern-0.4em Z$}} {\hbox{$\sf\scriptstyle
Z\kern-0.3em Z$}} {\hbox{$\sf\scriptscriptstyle Z\kern-0.2em
Z$}}}}

\newtheorem{theorem}{Theorem}
\newtheorem{lemma}[theorem]{Lemma}

\newtheorem{cor}[theorem]{Corollary}

\def\squareforqed{\hbox{\rlap{$\sqcap$}$\sqcup$}}
\def\qed{\ifmmode\squareforqed\else{\unskip\nobreak\hfil
\penalty50\hskip1em\null\nobreak\hfil\squareforqed
\parfillskip=0pt\finalhyphendemerits=0\endgraf}\fi}

\def\cA{{\mathcal A}}

\def\cI{{\mathcal I}}
\def\cJ{{\mathcal J}}

\def \sf {\mathfrak s}

\newcommand{\ignore}[1]{}

\def \F{\mathbb{F}}

\def \Z{\mathbb{Z}}

\def \Z{\mathbb{Z}}

\def\mand{\qquad\mbox{and}\qquad}

\def\\{\cr}
\def\({\left(}
\def\){\right)}

\begin{document}


\title[Exponential Pseudorandom Number Generator]{Periodic Structure of the Exponential Pseudorandom Number Generator}

 \author[J. Kaszi{\'a}n] {Jonas Kaszi{\'a}n}

\address{Department of Mathematics, RWTH Aachen, 
52056 Aachen, Germany}
\email{jonas.kaszian@rwth-aachen.de}

 \author[P.  Moree] {Pieter Moree}

\address{Max-Planck-Institut f{\"u}r Mathematik, Vivatsgasse 7, D-53111 Bonn, Germany}
\email{moree@mpim-bonn.mpg.de}

 \author[I. E. Shparlinski] {Igor E. Shparlinski}

\address{Department of Pure Mathematics, University of New South Wales,
Sydney, NSW 2052, Australia}
\email{igor.shparlinski@unsw.edu.au}

\begin{abstract} We investigate the periodic  structure of the
exponential pseudorandom number generator obtained from 
the map $x\mapsto g^x\pmod p$ that 
acts on the set $\{1, \ldots, p-1\}$. 
\end{abstract}

%

\maketitle

\section{Introduction}

\subsection{Motivation and our results}

Given a prime $p$ and an integer $g$  with $p\nmid g$ 
and an initial value $u_0\in \{1, \ldots, p-1\}$ we consider 
the sequence $\{u_n\}$ generated recursively by
\begin{equation}
\label{eq:seq u}
u_n \equiv g^{u_{n-1}} \pmod p,
\quad 1 \le u_n \le p-1, \qquad n = 1,2, \ldots\,,
\end{equation}
and then, for an integer parameter $k\ge 1$, we consider the sequence of 
integers $\xi^{(k)}_{n}\in \{0, \ldots, 2^k-1\}$ formed by the $k$ least significant 
bits of $u_n$, $n =0,1, \ldots$. This construction is called
the {\it exponential pseudorandom number generator\/} and 
has numerous cryptographic applications, see~\cite{FSS,Gen,GoldRos,Lag,PaSu,SJQ}
and references therein. 
Certainly, for the exponential pseudorandom number generator, 
as for any other pseudorandom number generator, 
the question of periodicity is of primal interest.

More precisely, the sequence $\{u_n\}$, as any other sequence
generated iterations of a function  on a finite set, becomes eventually periodic 
with some {\it cycle length\/} $t$. That is, there is some integer 
$s\ge 0$ such that 
\begin{equation}
\label{eq:cycle}
u_{n} = u_{n+t}, \qquad n = s, s+1,  \ldots.
\end{equation}
We always assume that $t$ is
the smallest positive integer with this property. Furthermore, 
the sequence $u_0, \ldots, u_{s+t-1}$ of length $\ell = s+t$,
where $t\ge 1$ and then $s\ge 0$ are chosen to be the 
smallest possible integers to satisfy~\eqref{eq:cycle}, is called
the {\it trajectory\/} of  $\{u_n\}$ and consists of the {\it tail\/}
$u_0, \ldots, u_{s-1}$ and the {\it cycle\/} $u_s, \ldots, u_{s+t-1}$. 

Clearly, we always have $\ell \le T$ where $T$ is the multiplicative
order of $g$ modulo $p$. 

Since  the sequence $\{u_n\}$ becomes eventually periodic with some 
cycle length $t$, so does the sequence $\{\xi^{(k)}_n\}$ and 
its cycle length $\tau_k$ divides $t$. 

We further remark that if $g$ is a primitive root modulo $p$, 
then the map $x\mapsto g^x\pmod p$  
acts bijectively on the set $\{1, \ldots, p-1\}$ or in other
words defines an element of the symmetric group $S_{p-1}$. Therefore, 
in this case the sequence $\{u_n\}$ is purely periodic, 
that is,~\eqref{eq:cycle} holds with $s=0$. This also means that 
in this case the sequence $\{\xi^{(k)}_n\}$ is purely periodic. 

As usual let $\varphi$ denote Euler's totient function. Recall that
there are exactly $\varphi(p-1)$ primitive roots modulo $p$.
The above map leads to precisely $\varphi(p-1)$ different elements
of $S_{p-1}$. The question is to what extent these $\varphi(p-1)$
permutations represent `generic permutations of $S_{p-1}$'. Note
that the cardinality $(p-1)!$ of $S_{p-1}$ is vastly larger than $\varphi(p-1)$ which
on average behaves as a constant times $p$.
   
Unfortunately there are essentially no theoretic results
about the behaviour of either of the sequences  $\{u_n\}$ and
$\{\xi^{(k)}_n\}$. 
In fact even the distribution of $t$ has not been properly 
investigated. 
If $g$ is a primitive root, which is the most interesting case
for cryptographic applications, then heuristically, the periodic 
behaviour of the sequence 
 $\{u_n\}$ can be modelled as  a random permutation
on the set $\{1, \ldots, p-1\}$, see~\cite{ABT} for a
wealth of results about random permutations. For example, 
by a result of~\cite{ShLl} one expects that 
$t = p^{1+o(1)}$ in this case.
If $g$ is not a primitive root it is not clear what
the correct 
statistical model describing the map  $x\mapsto g^x\pmod p$ 
should be. Probably, if $g$ is of order $T$ modulo $p$,
then one can further reduce  the residue $g^x\pmod p$ 
modulo $T$ and consider the associated permutation on the set 
$\{1, \ldots, T\}$ generated by the map  
$$
x\mapsto \(g^x\pmod p\)\pmod T. 
$$ 
This suggests that in this case one expects $t = T^{1 + o(1)}$, but 
the sequence  $\{u_n\}$ is not necessary purely periodic
anymore.

For the sequence  $\{\xi^{(k)}_n\}$ 
it is probably natural to expect that $\tau_k = t$ 
in the overwhelming 
majority of the cases (and for a wide range of values 
of $k$), but this question has not been properly addressed in the literature.

The only theoretic result here seems to be 
the bound of~\cite{FrSh} relating $t$ and $\tau_k$. 
First, as in~\cite[Section~5]{FrSh} we note that there
are at most 
$p2^{-k}+1$ 
integers $v \in \{1, \ldots, p-1\}$ with a given string 
of $k$  least significant bits. Hence, 
if $2^k < p$ then 
obviously 
\begin{equation}
\label{eq:triv bound}
\tau_k \ge  t 2^{k-1} /p. 
\end{equation} 
If $k \le (1/4 - \varepsilon) r$
for any fixed $\varepsilon >0$, 
where $r$ is the bit length of $p$, then it is shown in~\cite[Section~5]{FrSh}
that using bounds of exponential sums one can 
improve~\eqref{eq:triv bound} to 
\begin{equation}
\label{eq:FS bound}
\tau_k \ge c(\varepsilon)   t2^{2k}/p,
\end{equation} 
where $c(\varepsilon)>0$ depends only on $\varepsilon>0$. 
Clearly the bound~\eqref{eq:FS bound} trivially implies that for $k \ge r/4$ we have
\begin{equation}
\label{eq:FS bound large k}
\tau_k \ge   t p^{-1/2 + o(1)}, 
\end{equation} 
which however is weaker than~\eqref{eq:triv bound}
for $k \ge r/2$.

In this paper we use some results of~\cite{BGKS2} on the concentration of 
solutions of exponential congruences to 
sharpen~\eqref{eq:triv bound}, \eqref{eq:FS bound}
and~\eqref{eq:FS bound large k} 
for $k \ge (3/8 + \varepsilon) r$. 

We also use the same method to establish a lower bound for the number of 
distinct values in the sequence $\{\xi^{(k)}_n\}$.
Finally, we also show that for large values of $k$ the 
modern results on the sum-product problem (see~\cite{BuTs}) lead to better estimates.

Our results relate $\tau_k$ and $t$ and are meaningful 
only when $t$ is sufficiently large. 
Since no theoretic results about large values of 
$t$ are known, we study the behaviour of $t$ empirically. Our findings
are consistent with the map $x \mapsto g^x \pmod p$ having a generic
cycle structure. In particular, the results of our numerical tests 
exhibit a reasonable agreement with  those
predicted for random permutations, see~\cite{ABT}.

\subsection{Previously known results}

Here we briefly review several previously known results about the 
cycle structure of the map $x \mapsto g^x \pmod p$. Essentially only very short 
cycles, such as fixed points, succumb to the efforts of getting rigorous
results.

In particular, for an integer $k$ we denote by 
$N_{p,g}(k)$ the number of $u_0 \in \{1, \ldots, p-1\}$
such that for the sequence~\eqref{eq:seq u} we
have $u_k = u_0$. 
Note that $N_{p,g}(1)$ is the number of 
fixed points of the map $x \mapsto g^x \pmod p$.

The quantity $N_{p,g}(k)$  for $k=1,2,3$ has recently been
studied in~\cite{BKS1,BKS2,CobZah,GlShp,Hold,HoldMor1,HoldMor2,LPS,Zhang}. 
Fixed points with various restrictions on 
$u$ have been considered as well. 
For example, Cobeli and  Zaharescu~\cite{CobZah} have shown 
that  
\begin{equation*}
\begin{split}
\# \{(g, u) ~:~1\le g, u \le p-1,\ 
 &\gcd(u,p-1)=1, \ g^u \equiv u \pmod p\} \\
&  = \frac{\varphi(p-1)^2}{p-1}  + O\(\tau(p-1)p^{1/2}\log p\),
\end{split}
\end{equation*}
where $\tau(m)$ is the
number of positive integer divisors of $m\ge 1$. 
Unfortunately, the co-primality condition $\gcd(u,p-1)=1$ is 
essential for the method of~\cite{CobZah}, thus that result does not 
immediately extend to all $u \in  \{1, \ldots, p-1\}$.
Several more results and conjectures of similar flavour are 
presented by Holden and  Moree~\cite{HoldMor2}. 
Furthermore, an asymptotic formula for 
the average value $N_{p,g}(1)$ on average over $p$ and all primitive 
roots $g \in  \{1, \ldots, p-1\}$, as well as, 
over all $g \in  \{1, \ldots, p-1\}$ is
given by Bourgain, Konyagin and Shparlinski~\cite[Theorems~13 and~14]{BKS1}:
$$
\sum_{p\le Q} \frac{1}{p-1} 
\sum_{\substack{g=1\\g~\text{primitive root}}}^{p-1} N_{p,g}(1)  =(A+ o(1))\pi(Q)
$$
and 
$$
\sum_{p\le Q} \frac{1}{p-1} \sum_{g=1}^{p-1} N_{p,g}(1) =(1+ o(1))\pi(Q)
$$
as $Q\to \infty$, where
$$
A =   \prod_{p~\mathrm{prime}} \(1
-\frac{1}{p(p-1)}\) = 0.373955 \ldots
$$
is {\it Artin's constant\/} and, as usual,
$\pi(Q)$ is the number of primes
$p\le Q$.
It is also shown in~\cite[Theorem~11]{BKS2} that 
$$
 \sum_{g=1}^{p-1} N_{p,g}(1) = O(p),
$$
however, the conjecture by Holden and  Moree~\cite{HoldMor2}
that
\begin{equation}
\label{eq:asymp}
 \sum_{g=1}^{p-1} N_{p,g}(1) = (1+ o(1))p
\end{equation} 
remains open. It is known though that 
$$
 \sum_{g=1}^{p-1} N_{p,g}(1) \ge p + O(p^{3/4+o(1)}),
$$ 
see~\cite[Equation~(1.15)]{BKS2}. It is also shown 
in~\cite[Section~5.9]{BKS2} that~\eqref{eq:asymp}
may fail only on a very thin set of primes. 

It is also known that $N_{p,g}(1)\leq \sqrt{2p}+1/2$ for any
$g \in  \{1, \ldots, p-1\}$, see~\cite[Theorem~2]{GlShp}.

For $N_{p,g}(2)$, the only known result is the bound 
$$
 N_{p,g}(2) \le C(g) \frac{p}{\log p}
$$
of Glebsky and Shparlinski~\cite[Theorem~3]{GlShp}, where $C(g)$ depends on $g$.

Finally, by~\cite[Theorem~3]{GlShp} we have
$$
N_g(3)\leq \frac{3}{4}p+\frac{g^{2g+1}+g+1}{4} 
$$
(which is certainly a very weak bound).

 \section{Preparations}

\subsection{Density of points on exponential curves}

Let $p$ be a prime and $a$, $b$ and $g$  integers satisfying $p\nmid abg$.
Given two intervals $\cI$ and $\cJ$,  we denote by
 $R_{a,b,g,p}(\cI ,\cJ)$  the number 
of integer solutions of the system of congruences
\begin{align*}
au \equiv x \pmod p&\mand  
bg^u \equiv y \pmod p, \\\ (u,x,y)&\in\{1, \ldots, p-1\} \times \cI \times  \cJ .
\end{align*}

Upper bounds on  $R_{1,b,g,p}(\cI ,\cJ)$ 
are given  in~\cite[Theorems~23 and~24]{BGKS2}, which in turn
improve and generalise the previous estimates of~\cite{ChShp,CillGar}.
We need the following straightforward 
generalisations of the estimates of~\cite[Theorems~23 and~24]{BGKS2}
to an arbitrary $a$ with $p\nmid a$.

\begin{lemma}
\label{lem:RIJ} Suppose that $p\nmid ab$ and that $T$ is the multiplicative order of $g$ modulo
$p$. Let  $\cI$ and $\cJ$ be two intervals
consisting of $K$ and $L$ consecutive integers respectively,
where $L\le T$.
Then 
$$
R_{a,b,g,p}(\cI ,\cJ) \le \(\frac{K}{p^{1/3}L^{1/6}} + 1\)L^{1/2+o(1)}
$$
and 
$$
R_{a,b,g,p}(\cI ,\cJ) \le \(\frac{K}{p^{1/8}L^{1/6}} + 1\)L^{1/3+o(1)}.
$$
\end{lemma}

For intervals $\cI$ and $\cJ$ of the same length, we derive a more explicit form of Lemma~\ref{lem:RIJ}:

\begin{cor}
\label{cor:RIJ} Assume that $g$ is of multiplicative order $T$ modulo
$p$ and that $a$ and $b$ are integers such that $p\nmid ab$. 
Let  $\cI$ and $\cJ$ be two intervals
consisting of $H$ consecutive integers respectively,
where $H\le T$.
Then 
$$
R_{a,b,g,p}(\cI ,\cJ) \le H^{o(1)}
\left\{  \begin{array}{ll}
H^{1/3} ,& {\text{if}}\ H \le p^{3/20},\\
H^{7/6} p^{-1/8} , & {\text{if}}\ p^{3/20} < H \le p^{3/16},\\
H^{1/2} , & {\text{if}}\ p^{3/16} < H \le p^{2/5},\\
H^{4/3}p^{-1/3} , & {\text{if}}\ p^{2/5} < H.
\end{array} \right.
$$
\end{cor}

\subsection{Sum-product problem}

For a prime $p$, we denote by $\F_p$ the finite field of $p$ elements. 

Given a set $\cA\subseteq \F_p$ 
we define the sets
$$
2\cA =\{a_1+a_2 \ : \ a_1,  a_2 \in \cA\}
\quad \text{and} \quad 
\cA^2 =\{a_1\cdot a_2 \ : \ a_1,  a_2 \in \cA\}.
$$

The celebrated result of  Bourgain, Katz and Tao~\cite{BKT} asserts
that at least one of the cardinalities $\#\(\cA^2\)$ and 
$\#\(2\cA\)$ is always large.

The current state of affairs regarding  quantitative  versions of this result, 
due to several authors, has been summarised by 
Bukh  and Tsimerman~\cite{BuTs} as follows:

\begin{lemma}
\label{lem:AA} For an arbitrary set $\cA \subseteq \F_p$, we have
\begin{align*}
\max \{\#\(\cA^2\), \#\(2\cA\)\} & \\
\ge \(\# \cA \)^{o(1)} &
\left\{\begin{array}{ll}
\(\# \cA \)^{12/11},& \text{if $\# \cA \le p^{1/2}$,}\\
\(\# \cA \)^{7/6}p^{-1/24},&
 \text{if $p^{1/2}\le  \# \cA \le p^{35/68}$,}\\
\(\# \cA \)^{10/11}p^{1/11},&
 \text{if $p^{35/68}\le  \# \cA \le p^{13/24}$,}\\
 \(\# \cA \)^{2}p^{-1/2},&
 \text{if $p^{13/24}\le  \# \cA \le p^{2/3}$,}\\
 \(\# \cA\)^{1/2}p^{1/2},&
 \text{if $\# \cA \ge p^{2/3}$.}
\end{array}
\right.
\end{align*}
\end{lemma}

\section{Main Results} 

 \subsection{Period length}

For any $k \le r$ we now obtain an improvement of~\eqref {eq:triv bound}

\begin{theorem}
\label{thm:PRNG Per} 
For any $r$-bit prime $p$ and $g$ with  $p\nmid g$, we have 
$$
\tau_k \ge t p^{o(1)}
\left\{  \begin{array}{ll}
(2^k/p)^{1/3} ,& {\text{if}}\ k/r \ge 17/20 ,\\
2^{7k/6} p^{-25/24} , & {\text{if}}\   17/20  > k/r  \ge 13/16 ,\\
(2^k/p)^{1/2} , & {\text{if}}\ 13/16  > k/r  \ge 3/5 ,\\
2^{4k/3}p^{-1} , & {\text{if}}\ 3/5  >  k/r.
\end{array} \right.
$$
\end{theorem}

\begin{proof} Recall that we have the divisibility $\tau_k \mid t$ and consider the sequence
$u_{s\tau_k}$ for  $s = 1, \ldots, t/\tau_k$.
By the definition of $\tau_k$, all these numbers end with the same string of
$k$ least significant bits. Furthermore,  this  is also true for
$u_{s\tau_k + 1} \equiv g^{u_{s\tau_k}} \pmod p$.  Therefore, there 
are some integers $\lambda, \mu \in [0, 2^k-1]$ so that 
$$
u_{s\tau_k} = 2^k v_s + \lambda \mand u_{s\tau_k + 1} = 2^k w_s +\mu
$$
for some integers $v_s, w_s \in [0, 2^{r-k}-1]$.

Hence, defining $\alpha \in [1, p-1]$ by the congruence 
$\alpha  2^k \equiv 1 \pmod p$,  we see that the residues 
modulo $p$ of $\alpha  u_{s\tau_k}$ and of $\alpha  g^{u_{s\tau_k}}$
belong to some intervals of $\cI$ and $\cJ$, respectively, 
of length $2^{r-k}$ each. 
Since $t\le T$, where $T$ is
the multiplicative order of $g$, for these intervals $\cI$ and $\cJ$ we have
$$
t/\tau_k \le R_{\alpha , \alpha,g,p}(\cI ,\cJ) .
$$
Using Corollary~\ref{cor:RIJ} with $H = 2^{r-k}$,
we conclude the proof. 
\end{proof}

Combining Theorem~\ref{thm:PRNG Per}  with~\eqref{eq:FS bound}
and~\eqref{eq:FS bound large k}
we derive 

\begin{cor}
\label{cor:PRNG Per} 
For any $r$-bit prime $p$ and $g$ with  $p\nmid g$, we have 
$$
\tau_k \ge t p^{o(1)}
\left\{  \begin{array}{ll}
(2^k/p)^{1/3} ,& {\text{if}}\ k/r \ge 17/20 ,\\
2^{7k/6} p^{-25/24} , & {\text{if}}\   17/20  > k/r  \ge 13/16 ,\\
(2^k/p)^{1/2} , & {\text{if}}\ 13/16  > k/r  \ge 3/5 ,\\
2^{4k/3}p^{-1} , & {\text{if}}\ 3/5  >  k/r \ge 3/8 ,\\
p^{-1/2} , & {\text{if}}\ 3/8  >  k/r  \ge 1/4  ,\\
2^{2k}p^{-1} , & {\text{if}}\ 1/4  > k/r   . 
\end{array} \right.
$$
\end{cor}

\subsection{The number of distinct values}
\label{distinkt}
We now obtain a lower bound on the number $\nu_k(N)$ of distinct values 
which appear among the elements $\xi^{(k)}_n$, $n = 0, \ldots, N-1$. 
Let $\ell = s+t$ be the trajectory length of the sequence $\{u_n\}$,
see~\eqref{eq:cycle}. 

Note that if $2^k< p$ then the following analogue of~\eqref{eq:triv bound} holds:
\begin{equation}
\label{eq:triv bound val}
\nu_k(N) \ge N 2^{k-1}/p.
\end{equation}
In fact for $N = \ell= p^{1+o(1)}$
the  bound~\eqref{eq:triv bound val} is asymptotically 
optimal as we obviously have $\nu_k(N) \le 2^k$. However for 
smaller values of $\ell$ we obtain a series
of other bounds.

\begin{theorem}
\label{thm:Val Set 1} 
For any $r$-bit prime $p$ and $g$ with   $p\nmid g$, we have 
$$
\nu_k(N) \ge N^{1/2} p^{o(1)}
\left\{  \begin{array}{ll}
(2^k/p)^{1/6} ,& {\text{if}}\ 1 \ge k/r \ge 17/20 ,\\
2^{7k/12} p^{-25/48} , & {\text{if}}\   17/20  > k/r  \ge 13/16 ,\\
(2^k/p)^{1/4} , & {\text{if}}\ 13/16  > k/r  \ge 3/5 ,\\
2^{2k/3}p^{-1/2} , & {\text{if}}\ 3/5  >  k/r,
\end{array} \right.
$$
for all $N \le \ell$. 
\end{theorem}

\begin{proof}
Consider the pairs $(\xi^{(k)}_n, \xi^{(k)}_{n+1})$, 
$n = 0, \ldots, N-1$. Then at least one pair $(\lambda, \mu)$ 
appears at least $N/\nu_k^2(N)$ times. Since $N \le \ell< T$, where $T$ is
the multiplicative order of $g$, as in the proof of Theorem~\ref{thm:PRNG Per}
we obtain
$$
N/\nu_k^2(N) \le R_{\alpha , \alpha,g,p}(\cI ,\cJ) 
$$
for some intervals $\cI$ and $\cJ$ of length $2^{r-k}$ each and 
some integer $\alpha \in \{1, \ldots, p-1\}$. 
Using Corollary~\ref{cor:RIJ} with $H = 2^{r-k}$,
we conclude the proof. 
\end{proof}

Using the same technique as in~\cite[Section~5]{FrSh}, 
it is easy to show that any fixed pair
$(\lambda, \mu)$  occurs amongst the pairs 
$(\xi^{(k)}_n, \xi^{(k)}_{n+1})$, 
$n =0, \ldots, \ell-1$, at most $O\(p 2^{-2k} + p^{1/2} (\log p)^2\)$
times. So,  we also have 
$$
N/\nu_k^2(N) 
= O\(p 2^{-2k} + p^{1/2} (\log p)^2\), 
$$
and thus, after simple calculations,  we derive the following estimate.

\begin{cor}
\label{cor:Val Set} 
For any $r$-bit prime $p$ and any integer $g$ with  $p\nmid g$, we have 
$$
\nu_k(N) \ge 
N^{1/2} p^{o(1)}
\left\{  \begin{array}{ll}
(2^k/p)^{1/6} ,& {\text{if}}\ k/r \ge 17/20 ,\\
2^{7k/12} p^{-25/48} , & {\text{if}}\   17/20  > k/r  \ge 13/6 ,\\
(2^k/p)^{1/4} , & {\text{if}}\ 13/16  > k/r  \ge 3/5 ,\\
2^{2k/3}p^{-1/2} , & {\text{if}}\ 3/5  >  k/r \ge 3/8 ,\\
p^{-1/4} , & {\text{if}}\ 3/8  >  k/r  \ge 1/4  ,\\
2^{k}p^{-1/2} , & {\text{if}}\ 1/4  > k/r   ,
\end{array} \right.
$$
for all $N \le \ell$.
\end{cor}

We now obtain a different bound which is stronger than
Corollary~\ref{cor:Val Set} in a wide range of values of $k$ and $\ell$. 

\begin{theorem}
\label{thm:Val Set 2} 
For any $r$-bit prime $p$ and any integer $g$ with  $p\nmid g$, we have 
$$
\nu_k(N) \ge N^{o(1)} 
\left\{\begin{array}{ll}
N^{6/11}(2^k/p)^{1/2},& \text{if}\ N \le p^{1/2},\\
N^{7/12}2^{k/2} p^{-13/24},&
 \text{if}\ p^{1/2}<  N \le p^{35/68},\\
N^{5/11}2^{k/2} p^{-9/22},&
 \text{if}\ p^{35/68}< N \le p^{13/24},\\
 N 2^{k/2} p^{-1},&
 \text{if}\ p^{13/24} <  N \le p^{2/3},\\
N^{1/4}2^{k/2}p^{-1/4},&
 \text{if}\ N> p^{2/3},
\end{array}
\right.
$$
for all $N \le \ell$.
\end{theorem}

\begin{proof}
Consider the set $\cA = \{u_n\ : \ n =0, \ldots, N-1\}$.
Clearly $\#\cA = N$ as the first $N \le \ell$ elements of the sequence
$\{u_n\}$ are pairwise distinct.

Since $u_n = 2^{k} w_n + \xi^{(k)}_n$ for some integer $w_n \in [0, 2^{r-k}-1]$, 
$n = 0,1, \ldots$, 
we see that 
\begin{equation}
\label{eq:2A}
\#(2\cA) \le \nu_k^2(N) 2^{r-k+1}
\end{equation} 
(even if the addition of the elements of $\cA$ is considered in $\Z$ without the reduction
modulo $p$). 

Furthermore, from the definition of the sequence $\{u_n\}$  we see that
$$
\cA^2 = \{g^{a_1+a_2} \ : \ a_1,a_2 \in A\}
$$
(where $g^b$ is computed in $\F_p$), thus we also have 
\begin{equation}
\label{eq:A2}
\#(\cA^2) \le \nu_k^2(N) 2^{r-k+1}.
\end{equation} 

Comparing~\eqref{eq:2A} and~\eqref{eq:A2} with 
Lemma~\ref{lem:AA}, 
we conclude the proof. 
\end{proof}

In particular, if $N=  p^{1/2+o(1)}$ then 
Theorem~\ref{thm:Val Set 2} improves  Corollary~\ref{cor:Val Set} 
for $k \ge  (41/44+  \varepsilon) r$, with arbitrary $\varepsilon > 0$. 

\subsection{Frequency of values}
\label{freq}

We now give an upper bound on the frequency
$V_k(\omega)$ of a given $k$-bit string $\omega$ 
that appears in the full trajectory 
$\xi^{(k)}_n$, $n = 0, \ldots, \ell-1$. 

More precisely, let $\Omega_k(U)$ be the set of 
$k$-bit strings $\omega$ for which $V_k(\omega) \ge U$.

\begin{theorem}
\label{thm:Freq} 
For any $r$-bit prime $p$ and $g$ with   $p\nmid g$, we have 
$$
\# \Omega_k(U) \le   U^{-1} p^{o(1)}
\left\{  \begin{array}{ll}
2^{2k/3} p^{1/3} ,& {\text{if}}\ k/r \ge 17/20 ,\\
2^{k/6}  p^{25/24} , & {\text{if}}\   17/20  > k/r  \ge 13/16 ,\\
2^{k/2}p^{1/2} , & {\text{if}}\ 13/16  > k/r  \ge 3/5 ,\\
2^{-k/3}p , & {\text{if}}\ 3/5  >  k/r.
\end{array} \right.
$$
\end{theorem}

\begin{proof}
Consider the
pairs 
\begin{equation}
\label{eq:Bad pair}
(\xi^{(k)}_n, \xi^{(k)}_{n+1}), \qquad \xi^{(k)}_n \in  \Omega_k(U), \ n = 0, \ldots, \ell-1.
\end{equation} 
Clearly, there are
$$
W = \sum_{\omega\in \Omega_k(U)} V_k(\omega) \ge \# \Omega_k(U) U
$$
such pairs. 

Since  $\xi^{(k)}_{n+1}$ can take at most $2^k$ possible values,
we see that at least one pair $(\omega, \sigma)$ of two $k$-bit strings
occurs at least $W/2^k$ times amongst the pairs~\eqref{eq:Bad pair}.
Now, the same argument as used in the proof of Theorem~\ref{thm:PRNG Per}
implies that 
$$
W/2^k \le R_{\alpha , \alpha,g,p}(\cI ,\cJ) 
$$
for some intervals $\cI$ and $\cJ$ of lengths $2^{r-k}$ each and 
some integer $\alpha \in \{1, \ldots, p-1\}$. 
Using Corollary~\ref{cor:RIJ} with $H = 2^{r-k}$,
we conclude the proof. 
\end{proof}

Examining the value of $U$ for which the bound of  Theorem~\ref{thm:Freq} 
implies that $\# \Omega_k(U) < 1$, we derive 

\begin{cor}
\label{cor:PRNG Per} 
For any $r$-bit prime $p$ and $g$ with  $p\nmid g$, we have 
$$
V_k(\omega)\le    p^{o(1)}
\left\{  \begin{array}{ll}
2^{2k/3} p^{1/3} ,& {\text{if}}\ k/r \ge 17/20 ,\\
2^{k/6}  p^{25/24} , & {\text{if}}\   17/20  > k/r  \ge 13/16 ,\\
2^{k/2}p^{1/2} , & {\text{if}}\ 13/16  > k/r  \ge 3/5 ,\\
2^{-k/3}p , & {\text{if}}\ 3/5  >  k/r.
\end{array} \right.
$$
\end{cor}

\section{Numerical Results on Cycles in Exponential Map}

Here we present results of some numerical tests 
concerning the cycle structure of the permutation
on the set $\{1, \ldots, p-1\}$ generated by the map
$x \mapsto g^x \pmod p$. 

We use  $\cI_m$ to denote the dyadic interval $\cI_m = [2^{m-1}, 2^{m}-1]$. 

We test 500  pairs $(p,g)$ of primes $p$ and primitive roots $g$ 
modulo $p$ selected using a pseudorandom number generator 
separately each of the interval $p \in \cI_{20}$ and
$p \in \cI_{22}$ and $p \in \cI_{25}$.

We also repeat this for 60 pairs $(p,g)$ 
in the larger range $p \in \cI_{30}$. 

Let  $L_r(N)$ and $C(N)$  
be the length of the $r$th longest cycle and 
 the number of disjoint cycles
in a random permutation on  $N$ symbols, respectively.

We now recall that by the classical result of Shepp and  Lloyd~\cite{ShLl}
the ratios $\lambda_r(N) = L_r(N)/N$ 
is  expected to be 
$$
\lambda_r(N)=G_r+o(1),
$$ 
as $N\to \infty$, 
for some constants $G_r$, $r =1,2, \ldots$,
explicitly given in~\cite{ShLl} via some integral expressions. 
In particular, we find from~\cite[Table~1]{ShLl} that
$$
G_1=0.624329\ldots, \quad 
G_2=0.209580\ldots, \quad 
G_3= 0.088316\ldots,
$$
(we note that values reported in~\cite{KnTP} slightly deviate from those of~\cite{ShLl},
but they agree over the approximations given here). 
Interestingly, the constants $G_r$ also occur when one considers the size (in terms of number
of digits) of the $r$th largest prime factor of an integer $n$, see Knuth and 
Trabb Pardo~\cite{KnTP}.
For example, de Bruijn~\cite{deB} has shown that
$$\sum_{n\le x}\log P(n)= G_1 x \log x+O(x),$$
with $P(n)$ the largest prime factor of $n$, thus establishing a claim by Dickman.
The constant $G_1$ is now known as the Golomb-Dickman constant. For further information and
references see the book by Finch~\cite[Section~5.4]{Finch}.

We also recall that Goncharov~\cite{Gonch} 
has shown that the ratio $\gamma(N) = C(N)/\log N$, 
is  expected  to be 
$$
\gamma(N) = 1+o(1)\qquad \text{as}\ N\to \infty.
$$

The above asymptotic results can also be found in~\cite[Section~1.1]{ABT}.

In Table~\ref{tab:Cycle} we present the average 
value, over the tested primes $p$ in each group, 
of the lengths of the 1st, 2nd and 3rd longest 
cycles normalised by dividing by the size of the 
set, that is, by $p-1$. 

We also calculate the number of cycles
for the above pairs $(p,g)$, normalised by dividing by $\log(p-1)$,
and then present the average value for each of the ranges.

\begin{table}[H]
  \centering
  \begin{tabular}{|l|l|l|l|l|l|}
\hline
Range & $\cI_{20}$ & $\cI_{22}$ & $\cI_{25}$ & $\cI_{30}$  \\
 \# of   $(p,g)$ &500 &500 &500 &60\\
\hline
Aver. $\lambda_1$ &0.63946789 &	0.61508766	& 0.63157252 & 0.60441217
 \\
Aver. $\lambda_2$ & 	0.19999487 & 	0.21687612 & 0.20469932 & 0.21715242
 \\
Aver. $\lambda_3$  & 0.08646438 & 0.08450844 & 0.09092497	 & 0.09354165\\
Aver. $\gamma$& 1.03813497 & 1.03324650 & 1.03014896 & 1.05566909 \\
\hline
\end{tabular}
\caption{Numbers of connected components}
  \label{tab:Cycle}
\end{table}

We note that we have also tried to compare the length of the smallest cycle with 
the expected length $e^{-\gamma} \log p$ 
for a random permutation on $\{1, \ldots, p-1\}$,  
where $\gamma =0.5772\ldots$ is the
Euler-Mascheroni constant. However the results are inconclusive 
and require further tests and investigation.

\section{Comments}

It is certainly interesting to study  similar questions over 
arbitrary finite fields, although in this case there is no
canonical  way to interpret field elements as integer numbers 
and thus to extract bits from field elements. 
Probably the most interesting and natural case is the
case of binary fields $\F_{2^r}$ of $2^r$ elements with a
sufficiently large $r$. First, we use the isomorphism 
$\F_{2^r} = \F_2(\alpha)$, where $\alpha$ is a root 
of an irreducible polynomial over $\F_2$ of degree $r$. 
Now we can represent each element of $\F_{2^r}$ as an $r$-dimensional
binary vector of coefficients in the basis $1, \alpha, \ldots, \alpha^{r-1}$,
and the bit extraction is now apparent. For example, the proof of~\cite[Theorem~2]{GlShp}
can easily be adjusted to give a square-root bound for the number 
of fixed points (when we identify elements of $\F_{2^r}$ with $r$-dimensional
binary vectors). 
It is also quite likely that using the
results and methods of~\cite{CillShp} one can obtain some variants of our 
results in these settings. 

Furthermore, for cryptographic applications it is also interesting 
to study the relation between $t$ and $\tau_k$ and, in particular,  obtain  
improvements of Corollaries~\ref{cor:Val Set}  and~\ref{cor:PRNG Per}
for almost all $p$ and almost all initial values $u_0$. 
It is quite likely that the method of~\cite{BGKS3}, 
combined with the ideas of~\cite{BGKS2}, can 
be used to derive such results.

Finally we note that exponential maps have also been considered 
modulo prime powers,  see~\cite{Gleb,HoldRob}. Although many
computational problems,
such as the discrete logarithm problem, are easier modulo prime powers,
the corresponding exponential pseudorandom number generator does not 
seem to have any immediate weaknesses.

\section*{Acknowledgements}

The authors would like to thank Daniel Panario for
useful discussions and references and to Arne Winterhof 
for a careful reading of the manuscript.

This work was finished during a very enjoyable internship
of the first author and research stay of the third author at the Max
Planck Institute
for Mathematics, Bonn.
The third author was also
supported in part by ARC grants DP110100628 and
DP130100237.

\end{document}